\documentclass[11pt,a4paper]{article}
\usepackage[T1]{fontenc}
\usepackage{lmodern}
\usepackage{amsmath,amssymb,amsthm,mathtools}
\usepackage[margin=28mm,headheight=14pt]{geometry}
\usepackage{microtype}
\usepackage{enumitem}
\usepackage{xcolor}
\usepackage{fancyhdr}
\usepackage[colorlinks=true,linkcolor=blue!45!black,citecolor=blue!45!black,
  urlcolor=blue!45!black,pdftitle={A candidate proof of the Latala--Strzelecka bound},
  pdfsubject={Research draft: Conjecture 5 and a coupled chaining construction}]{hyperref}

\newcommand{\E}{\mathbb E}
\newcommand{\Prob}{\mathbb P}
\newcommand{\R}{\mathbb R}

\newcommand{\Log}{\operatorname{Log}}
\newcommand{\diam}{\operatorname{diam}}
\newcommand{\op}{\mathrm{op}}
\newcommand{\HS}{\mathrm{HS}}

\newcommand{\ga}{\gamma_{2,2}}
\newcommand{\Bplus}[1]{B_{2,+}^{#1}}
\newcommand{\norm}[1]{\left\lVert #1\right\rVert}

\newcommand{\less}{\lesssim}

\theoremstyle{plain}
\newtheorem{theorem}{Theorem}[section]
\newtheorem{proposition}[theorem]{Proposition}
\newtheorem{lemma}[theorem]{Lemma}

\theoremstyle{remark}
\newtheorem{remark}[theorem]{Remark}
\numberwithin{equation}{section}
\setlist[enumerate]{leftmargin=*,itemsep=3pt,topsep=5pt}
\allowdisplaybreaks[1]

\title{A candidate proof of the sharp\\
Lata\l a--Strzelecka Gaussian matrix bound}
\author{{Witold Bednorz, Rafal Martynek and Rafal Meller}
\footnote{Research partially supported by  Grant UMO-2022/47/B/ST1/02114}
\footnote{Institute of Mathematics, University of Warsaw, Banacha 2, 02-097 Warszawa, Poland}}
\date{}

\begin{document}
\maketitle
\thispagestyle{empty}
\begin{abstract}
We give a candidate proof of Conjecture~5 in the third arXiv version of
Lata\l a and Strzelecka's \emph{Operator $\ell_p\to\ell_q$ norms of Gaussian
matrices}. The argument uses a positive power parametrization of the original
canonical image. Its derivative profiles are controlled by four compatible
chaining hierarchies, including two hierarchies for coupled row and column
lengths. A localization argument then transfers this squared chaining control
to the canonical Gaussian metric. The resulting constants are uniform in the
dimensions and in $1\le p\le2\le q\le\infty$. We also record a precise
prefix-dependent block representation with the corrected operator budget.

\smallskip
\noindent\textbf{Status.} This is a self-contained research draft of the proposed
argument, not a claim of independent verification or a published resolution.
The new profile estimate and transfer argument are proved below; standard
chaining inputs are stated explicitly with references.
\end{abstract}

\section{Statement and notation}
Let $A=(a_{ij})_{i\le m,j\le n}$ be a real deterministic matrix, and let
$g_{ij}$ be independent standard Gaussian variables. Write
$G_A=(a_{ij}g_{ij})$. For $1\le p\le2\le q\le\infty$, put
\[
 r=p^*,\qquad s=q,\qquad
 D_r=\max_{i\le m}\norm{(a_{ij})_{j\le n}}_r,\qquad
 D_s=\max_{j\le n}\norm{(a_{ij})_{i\le m}}_s.
\]
Here $p^*$ is the conjugate exponent, and a norm with exponent $\infty$ is a
maximum. Set
\begin{equation}\label{eq:M}
 S^*=\E\max_{i,j}|a_{ij}g_{ij}|,\qquad
 M=\sqrt{r\wedge\Log n}\,D_r+
   \sqrt{s\wedge\Log m}\,D_s+S^*,
\end{equation}
where $\Log d=1\vee\log d$ for $d\ge1$. All unspecified constants below are
absolute. They do not depend on $r,s,m,n$, or $A$.

\begin{theorem}[Target bound]\label{thm:main}
For every $1\le p\le2\le q\le\infty$ and every finite real matrix $A$,
\begin{equation}\label{eq:target}
 \E\norm{G_A}_{p\to q}\less M.
\end{equation}
\end{theorem}

This is the upper estimate formulated as Conjecture~5 in
\cite[Conjecture~5]{LS}. It is stronger than a bound whose implicit constant
may depend arbitrarily on $p$ and $q$. The argument below proves the target
bound through the two estimates
\begin{equation}\label{eq:roadmap}
 \ga(K,\rho)\less M,
 \qquad
 \gamma_2(K,d)\less R+\ga(K,\rho),
 \qquad R\less M.
\end{equation}
The set $K$, the canonical metric $d$, and the profile metric $\rho$ are defined
in Section~\ref{sec:geometry}. No change of the canonical image is made in
these estimates.

\section{Chaining conventions and standard inputs}\label{sec:tools}
For a pseudometric space $(T,d)$, an admissible sequence is a sequence of
nested partitions $(\mathcal A_k)_{k\ge0}$ such that
$\mathcal A_0=\{T\}$ and $|\mathcal A_k|\le2^{2^k}$. Define
\begin{equation}\label{eq:gamma}
 \gamma_{a,b}(T,d)=\inf_{(\mathcal A_k)}
 \sup_{x\in T}\left(
   \sum_{k\ge0}2^{kb/a}\diam_d(\mathcal A_k(x))^b
 \right)^{1/b},
 \qquad \gamma_a=\gamma_{a,1}.
\end{equation}
Zero-distance points may be identified. Let $e_k(T,d)$ denote the infimum
of radii of covers by at most $2^{2^k}$ balls. Allowing centers in the ambient
space or requiring centers in $T$ changes estimates by at most a factor two.
Likewise, the conventions with strict cardinality inequalities are equivalent
after a fixed shift of the level.

We use the following established facts. The forms stated here are sufficient
for the rest of the manuscript.
\begin{enumerate}[label=\textup{(\roman*)}]
\item\label{tool:gaussian} For a centered Gaussian process with canonical metric
$d$, the Gaussian chaining upper bound is
\begin{equation}\label{eq:gaussian}
 \E\sup_{x\in T}(X_x-X_{x_0})\less\gamma_2(T,d).
\end{equation}
\item\label{tool:contraction} The contraction principle
\cite[Theorem~3.1]{vH} says that, for a sufficiently small absolute
$\theta>0$, the inequalities
\begin{equation}\label{eq:contraction-hyp}
 e_k(B,d)\le\theta\diam_d B+\sup_{x\in B}s_k(x)
 \quad(k\ge0,\ B\subset T)
\end{equation}
imply
\begin{equation}\label{eq:contraction}
 \gamma_2(T,d)\less\sup_{x\in T}
   \sum_{k\ge0}2^{k/2}s_k(x).
\end{equation}
\item\label{tool:ellipsoid} For a linear map $T$ between finite-dimensional
Euclidean spaces, Hilbert ellipsoid entropy gives
\begin{equation}\label{eq:ellipsoid}
 \sum_{k\ge0}2^k e_k(TB_2,\norm{\cdot}_2)^2
 \less\norm{T}_{\HS}^2.
\end{equation}
This is the entropy estimate used in \cite[Section~7.3]{vH}.
\item\label{tool:ball} If $N$ is a seminorm on $\R^d$ and $g$ is a standard
Gaussian vector, then
\begin{equation}\label{eq:ball-chain}
 \ga(B_2^d,N)\less\E N(g).
\end{equation}
Indeed, the uniform convexity estimate
\cite[Theorem~5.8]{vH} bounds the left side by
$C\sup_k2^{k/2}e_k(B_2^d,N)$; dual Sudakov bounds this supremum by
$C\E N(g)$. A seminorm follows by a limit from norms.
\end{enumerate}

The Gaussian upper bound and the entropy facts are classical; they are also
discussed in \cite{vH}. One additional input is the chaining lower bound for
canonical symmetric exponential processes in \cite[Theorem~1]{LT}. Its exact
consequence needed here is proved next.

\begin{lemma}[Positive exponential chaining]\label{lem:positive}
Let $L\subset\R_+^d$ be compact, and let $E_1,\ldots,E_d$ be independent
exponential variables of mean one. Then
\begin{equation}\label{eq:positive}
 \ga(L,\norm{\cdot}_\infty)^2
 \less\E\sup_{\ell\in L}\sum_{i=1}^d\ell_i^2E_i.
\end{equation}
\end{lemma}
\begin{proof}
Write $Q=\{(\ell_i^2)_i:\ell\in L\}$. Nonnegativity gives
\[
 |x-y|^2\le|x^2-y^2|\quad(x,y\ge0).
\]
Pulling partitions of $Q$ back to $L$ therefore yields
\begin{equation}\label{eq:square-pullback}
 \ga(L,\norm{\cdot}_\infty)^2
 \le\gamma_1(Q,\norm{\cdot}_\infty).
\end{equation}
Let $Z_i=\varepsilon_iE_i$, where the $\varepsilon_i$ are independent
Rademacher signs, also independent of the $E_i$. For $h\in\R^d$ and $t\ge1$,
conditional Jensen's inequality, applied after fixing $Z_i$, gives
\[
 \norm{\sum_jh_jZ_j}_{L_t}\ge |h_i|\norm{Z_i}_{L_t}
 \gtrsim t|h_i|.
\]
Consequently $\norm{\sum_jh_jZ_j}_{L_t}\gtrsim t\norm h_\infty$.
The exponential variables satisfy the regular moment assumptions of
\cite[Theorem~1]{LT}, with absolute parameters; variance normalization only
rescales by an absolute factor. That theorem bounds the partition functional
formed with the $L_{2^k}$ increment metrics by the expected range of the
process. It follows that
\[
 \gamma_1(Q,\norm{\cdot}_\infty)
 \less\E\sup_{q,q'\in Q}\sum_i(q_i-q'_i)Z_i
 =2\E\sup_{q\in Q}\sum_iq_iZ_i.
\]
Since $q_i\ge0$, this last supremum is at most
$\sup_{q\in Q}\sum_iq_iE_i$ pointwise. Together with
\eqref{eq:square-pullback}, this proves the claim. All processes here are
continuous on compact finite-dimensional sets, so the usual separability
convention causes no additional issue.
\end{proof}

\section{The canonical image and derivative profiles}\label{sec:geometry}
The zero matrix is immediate, so it may be excluded throughout the argument.
Replacing $a_{ij}$ by $|a_{ij}|$ does not change the distribution of $G_A$.
The canonical coefficient images are related by the fixed Euclidean isometry
that multiplies coordinate $(i,j)$ by the sign of $a_{ij}$ (take sign one
at a zero entry). Every image representation below can therefore be
transported back to the original real matrix. We henceforth suppose
$a_{ij}\ge0$. Put
\begin{equation}\label{eq:power-map}
 K=\Bplus m\times\Bplus n,\qquad
 \alpha=2-\frac2s,\qquad\beta=2-\frac2r,
 \qquad F(u,v)_{ij}=a_{ij}u_i^\alpha v_j^\beta.
\end{equation}
Both exponents belong to $[1,2]$. Powers are taken coordinatewise.
Since $\alpha s^*=2$ and $\beta p=2$, the maps $u\mapsto u^\alpha$ and
$v\mapsto v^\beta$ are onto the positive parts of $B_{s^*}^m$ and
$B_p^n$, respectively. Thus $F(K)$ is exactly the positive canonical image.

For $z=(u,v)\in K$, define
\begin{equation}\label{eq:profiles}
 U_{ij}(z)=\alpha a_{ij}u_i^{\alpha-1}v_j^\beta,
 \qquad
 V_{ij}(z)=\beta a_{ij}u_i^\alpha v_j^{\beta-1}.
\end{equation}
We use $t^0=1$, including at $t=0$. This is the derivative convention for a
linear coordinate. The derivative $T_z=DF(z)$ acts by
\[
 (T_z(h,k))_{ij}=U_{ij}(z)h_i+V_{ij}(z)k_j.
\]
The metrics used throughout are
\begin{align}
 d(z,z')&=\norm{F(z)-F(z')}_2,\label{eq:d}\\
 \rho(z,z')&=
   \max_i\norm{(U_{ij}(z)-U_{ij}(z'))_j}_2
  +\max_j\norm{(V_{ij}(z)-V_{ij}(z'))_i}_2.\label{eq:rho}
\end{align}
The target norm in \eqref{eq:d} is the Euclidean, or Frobenius, norm.
These are continuous pseudometrics. In particular,
\begin{equation}\label{eq:derivative-op}
 \norm{T_z-T_{z'}}_\op\le\rho(z,z').
\end{equation}

\begin{lemma}[Trace bound]\label{lem:trace}
With $R=\sup_{z\in K}\norm{T_z}_\HS$, one has
\begin{equation}\label{eq:trace}
 R^2\le4(D_r^2+D_s^2).
\end{equation}
Moreover $\sup_{z\in K}\norm{F(z)}_2\le a_*:=\max_{i,j}a_{ij}\le R$.
\end{lemma}
\begin{proof}
For fixed $j$, H\"older's inequality gives
\[
 \sum_i a_{ij}^2u_i^{2-4/s}
 \le\norm{(a_{ij})_i}_s^2\qquad(u\in\Bplus m).
\]
At $s=2$ this is an equality independent of $u$, under the zero-power
convention; at $s=\infty$ it is the maximum bound. Since
$\sum_jv_j^{2\beta}\le1$, it follows that
$\sum_{i,j}U_{ij}(z)^2\le\alpha^2D_s^2$.
The transposed argument gives
$\sum_{i,j}V_{ij}(z)^2\le\beta^2D_r^2$.
These sums add to $\norm{T_z}_\HS^2$, proving \eqref{eq:trace}.
Also
\[
 \norm{F(u,v)}_2^2\le a_*^2
   \Big(\sum_i u_i^{2\alpha}\Big)
   \Big(\sum_j v_j^{2\beta}\Big)\le a_*^2.
\]
Evaluating a derivative at $u=e_i,v=e_j$ for a largest entry shows
$a_*\le R$.
\end{proof}

\section{Gaussian moments for common-coordinate profiles}\label{sec:moments}
Let $(g_i)_i$ and $(g_j)_j$ denote standard Gaussian vectors of the indicated
dimensions. Define
\begin{equation}\label{eq:H}
 H_s=\left(\E\Big[\max_j\norm{(a_{ij}g_i)_i}_s\Big]^2\right)^{1/2},
 \qquad
 H_r=\left(\E\Big[\max_i\norm{(a_{ij}g_j)_j}_r\Big]^2\right)^{1/2}.
\end{equation}
The same coordinate Gaussian is used across all columns in $H_s$, and across
all rows in $H_r$. Independence between these row or column norms is not
assumed.

\begin{lemma}\label{lem:moments}
Uniformly for $r,s\in[2,\infty]$,
\begin{equation}\label{eq:H-bound}
 H_r\less\sqrt{r\wedge\Log n}\,D_r+S^*,
 \qquad
 H_s\less\sqrt{s\wedge\Log m}\,D_s+S^*.
\end{equation}
\end{lemma}
\begin{proof}
It suffices to prove the row estimate. First take $2\le t<\infty$ and put
$f_i(g)=\norm{(a_{ij}g_j)_j}_t$, $b_i=\max_j a_{ij}$, and
$D_t^{\mathrm{row}}=\max_i\norm{(a_{ij})_j}_t$.
Because $t\ge2$, $f_i$ is $b_i$-Lipschitz for Euclidean distance. Also
\begin{equation}\label{eq:mean-row}
 \E f_i\le(\E f_i^t)^{1/t}
 =\norm{g_1}_{L_t}\norm{(a_{ij})_j}_t
 \less\sqrt t\,D_t^{\mathrm{row}}.
\end{equation}
Discard zero $b_i$ and reorder the remaining values decreasingly. For
independent standard Gaussians $\xi_i$, set
$B=\E\max_i b_i|\xi_i|$. The elementary lower bound for the maximum of
$k$ Gaussian absolute values gives
\begin{equation}\label{eq:ordered}
 B\gtrsim b_k\sqrt{\log(k+1)}\quad(k\ge1).
\end{equation}
Choosing one maximizing entry in each row gives independent entries of the
original array; hence $B\le S^*$.

Gaussian concentration and a union bound, which do not require independence
of the $f_i$, now give, for all sufficiently large absolute $x$,
\begin{align*}
 \Prob\left\{\max_i(f_i-\E f_i)>xB\right\}
 &\le\sum_i\exp\left(-\frac{x^2B^2}{2b_i^2}\right)\\
 &\le\sum_{i\ge1}(i+1)^{-c x^2}
 \le C e^{-c' x^2}.
\end{align*}
Integrating this tail and using \eqref{eq:mean-row} yields
\begin{equation}\label{eq:unsaturated}
 \norm{\max_i f_i}_{L_2}\less\sqrt t\,D_t^{\mathrm{row}}+S^*.
\end{equation}

If $t\ge\Log n$, including $t=\infty$, then
\[
 \max_i\norm{(a_{ij}g_j)_j}_t
 \le e\max_{i,j}a_{ij}|g_j|
 =e\max_j c_j|g_j|,\qquad c_j=\max_i a_{ij}.
\]
The last maximum has $L_2$ norm at most an absolute multiple of its mean.
For example, Gaussian concentration bounds its variance by
$(\max_jc_j)^2$, while its mean is at least
$\sqrt{2/\pi}\max_jc_j$. Choosing one entry in each column shows that this
mean is at most $S^*$. Thus the row maximum has $L_2$ norm at most $CS^*$
in this range.
Combining this observation with \eqref{eq:unsaturated} proves the first
inequality in \eqref{eq:H-bound}. Transposition proves the second.
\end{proof}

\section{A single hierarchy for the coupled profiles}\label{sec:coupling}
The following estimate is the main profile bound.
\begin{proposition}\label{prop:profile}
For the metric \eqref{eq:rho},
\begin{equation}\label{eq:profile-main}
 P:=\ga(K,\rho)\less H_r+H_s\less M.
\end{equation}
\end{proposition}

Define the coupled lengths
\begin{align}
 L_i(u,v)&=u_i^{\alpha-1}
       \left(\sum_j a_{ij}^2v_j^{2\beta}\right)^{1/2},\label{eq:L}\\
 C_j(u,v)&=v_j^{\beta-1}
       \left(\sum_i a_{ij}^2u_i^{2\alpha}\right)^{1/2}.\label{eq:C}
\end{align}
The notation $C_j$ here denotes a length, not an unspecified constant.
We also use the fixed seminorms
\begin{equation}\label{eq:Ns}
 N_r(h)=\max_i\norm{(a_{ij}h_j)_j}_r,
 \qquad N_s(k)=\max_j\norm{(a_{ij}k_i)_i}_s.
\end{equation}

\begin{lemma}[Profile decomposition]\label{lem:polar}
For $z=(u,v),z'=(u',v')\in K$,
\begin{align}\label{eq:profile-domination}
 \rho(z,z')\le{}&\alpha\norm{L(z)-L(z')}_\infty
       +\beta\norm{C(z)-C(z')}_\infty\notag\\
 &+2\alpha\beta\big(N_r(v-v')+N_s(u-u')\big).
\end{align}
\end{lemma}
\begin{proof}
For vectors $x,y$ in any normed space and $a,b\ge0$,
\begin{equation}\label{eq:polar}
 \norm{ax-by}\le
   |a\norm x-b\norm y|+2\min(a,b)\norm{x-y}.
\end{equation}
Indeed, if $a\ge b$, split $ax-by=(a-b)x+b(x-y)$ and use the reverse
triangle inequality to bound $(a-b)\norm x$ by
$|a\norm x-b\norm y|+b\norm{x-y}$. The other case is symmetric.

For the straight segment $v_t=(1-t)v+tv'$ in $\Bplus n$, H\"older gives
\[
 \norm{(a_{ij}h_jv_{t,j}^{\beta-1})_j}_2
 \le\norm{(a_{ij}h_j)_j}_r,
 \qquad \frac12=\frac1r+\frac{\beta-1}{2}.
\]
This includes $r=2$ and $r=\infty$ with the stated conventions. Integrating
the derivative along the segment gives
\begin{equation}\label{eq:power-difference}
 \max_i\norm{(a_{ij}(v_j^\beta-v_j'^\beta))_j}_2
 \le\beta N_r(v-v').
\end{equation}
Apply \eqref{eq:polar} to each row, with $a=u_i^{\alpha-1}$,
$b=u_i'^{\alpha-1}$, $x=(a_{ij}v_j^\beta)_j$, and
$y=(a_{ij}v_j'^\beta)_j$. Since $a,b\le1$, the row part of $\rho$ is
bounded by $\alpha\norm{L-L'}_\infty+2\alpha\beta N_r(v-v')$.
The column argument gives the remaining terms.
\end{proof}

\begin{lemma}[Chaining the lengths]\label{lem:lengths}
One has
\begin{equation}\label{eq:length-chains}
 \ga(L(K),\norm{\cdot}_\infty)\less H_s,
 \qquad
 \ga(C(K),\norm{\cdot}_\infty)\less H_r.
\end{equation}
\end{lemma}
\begin{proof}
For fixed nonnegative $E_i$ and $v$, optimize over $u$ to obtain
\begin{align}
 \sup_{u\in\Bplus m}\sum_iE_iL_i(u,v)^2
 &=\sup_{u\in\Bplus m}
    \sum_i u_i^{2-4/s}E_i\sum_j a_{ij}^2v_j^{2\beta}\notag\\
 &=\norm{\left(E_i\sum_j a_{ij}^2v_j^{2\beta}\right)_i}_{s/2}.
 \label{eq:opt-u}
\end{align}
At $s=2$ the last norm is $\ell_1$, and at $s=\infty$ it is $\ell_\infty$.
Minkowski's inequality and $\sum_jv_j^{2\beta}\le1$ show that the supremum
over $v$ equals
\begin{equation}\label{eq:coupled-sup}
 \sup_{u,v}\sum_iE_iL_i(u,v)^2
 =\max_j\norm{(E_i a_{ij}^2)_i}_{s/2}
 =\max_j\norm{(a_{ij}\sqrt{E_i})_i}_s^2.
\end{equation}
Equality is achieved by taking $v$ to be a maximizing coordinate vector.

Take independent standard Gaussian vectors $g,h$. The independent variables
$(g_i^2+h_i^2)/2$ have the exponential distribution of mean one. Pointwise,
\[
 \max_j\norm{(a_{ij}\sqrt{(g_i^2+h_i^2)/2})_i}_s
 \le\frac1{\sqrt2}\left(
   \max_j\norm{(a_{ij}|g_i|)_i}_s+
   \max_j\norm{(a_{ij}|h_i|)_i}_s\right).
\]
The $L_2$ norm of the right side is at most $\sqrt2 H_s$.
Lemma~\ref{lem:positive} applied to \eqref{eq:coupled-sup} proves the
first assertion. The transposed calculation proves the second.
\end{proof}

\begin{proof}[Proof of Proposition~\ref{prop:profile}]
In addition to the two chains from Lemma~\ref{lem:lengths},
\eqref{eq:ball-chain} gives
\[
 \ga(\Bplus n,N_r)\less\E N_r(g)\le H_r,
 \qquad
 \ga(\Bplus m,N_s)\less\E N_s(g)\le H_s.
\]
Choose nested admissible partitions for these four sets, within a fixed
factor of their infima. Pull them back to $K$ by $L,C,v,u$, respectively.
At level $k\ge2$, intersect their cells from level $k-2$. The intersection
partition has at most
\[
 \big(2^{2^{k-2}}\big)^4=2^{2^k}
\]
cells, and the sequence is nested. Use $\{K\}$ at the initial levels.
The inequality $(x_1+\cdots+x_4)^2\le4\sum_{a=1}^4x_a^2$, together with
\eqref{eq:profile-domination}, bounds the branchwise sum of squared
$\rho$-diameters by $C(H_r^2+H_s^2)$. The initial diameters are bounded by
the same four chain budgets. Lemma~\ref{lem:moments} completes the proof.
\end{proof}

\begin{remark}\label{rem:coupling}
The label for $L$ refines the entire vector
$u_i^{\alpha-1}(\sum_j a_{ij}^2v_j^{2\beta})^{1/2}$. It does not refine
$u^{\alpha-1}$ separately from $v$. Thus the dependence that matters when
$r$ and $s$ are very different is retained in every cell. All four labels are
pulled back to the original set $K$ and then intersected, so no independent
choice of incompatible row and column bodies is introduced.
\end{remark}

\section{Transferring profile refinements to the canonical metric}
\label{sec:transfer}
\begin{proposition}[Transfer estimate]\label{prop:transfer}
For the map and metrics in Section~\ref{sec:geometry},
\begin{equation}\label{eq:transfer}
 \gamma_2(K,d)\less R+P,
 \qquad P=\ga(K,\rho),\quad
 R=\sup_{z\in K}\norm{T_z}_\HS.
\end{equation}
\end{proposition}
We give the localization and entropy argument in detail. In particular, the
local tangent map may change from cell to cell, while its entropy is charged
to a fixed original point when the levels are summed.

\subsection{Geometric midpoints and endpoint remainders}
Write
$\mathcal E(u,v)=\norm u_2^2+\norm v_2^2\le2$.
For $x,y\in K$, let $m(x,y)$ be their coordinatewise geometric mean. Then
$m(x,y)\in K$ and
\begin{equation}\label{eq:energy-mid}
 \mathcal E(m(x,y))
 =\frac{\mathcal E(x)+\mathcal E(y)-\norm{x-y}_2^2}{2}.
\end{equation}
Each entry of $F(m)$, $U(m)$, and $V(m)$ is the geometric mean of the
corresponding endpoint entries. Hence it lies between them, and
\begin{equation}\label{eq:mid-contract}
 d(x,m(x,y))\le d(x,y),\qquad
 \rho(x,m(x,y))\le\rho(x,y).
\end{equation}

\begin{lemma}[Endpoint-controlled linearization]\label{lem:remainder}
For every $x,y,b\in K$,
\begin{equation}\label{eq:remainder}
 \norm{F(x)-F(y)-T_b(x-y)}_2
 \le2\big(\rho(x,b)+\rho(y,b)\big)\norm{x-y}_2.
\end{equation}
\end{lemma}
\begin{proof}
First suppose all endpoint coordinates are positive. Join $y$ to $x$ by
the coordinatewise geometric path $w(t)=x^t y^{1-t}$, $0\le t\le1$.
H\"older's inequality shows that this path stays in $K$. Every coordinate of
the path is monotone, and every entry of either derivative profile stays
between its endpoint entries.

Write $x=(u,v)$ and $y=(u',v')$. After subtracting $T_b(x-y)$, the absolute
value of the contribution to coordinate $(i,j)$ from variation of $u_i$ is
at most
\[
 |u_i-u_i'|\big(|U_{ij}(x)-U_{ij}(b)|+
                         |U_{ij}(y)-U_{ij}(b)|\big).
\]
Its Frobenius norm is bounded by $\norm{u-u'}_2$ times the sum of the two
row-profile distances. The $v$ contribution is bounded by
$\norm{v-v'}_2$ times the corresponding sum of column-profile distances.
The triangle inequality gives \eqref{eq:remainder}. General endpoints follow
by approximation by positive points in $K$, since $F,U,V$ are continuous
under our zero-power convention.
\end{proof}

\subsection{A regularizer that pays both residual errors}
Fix $a>0$ and write $t_k=a2^{k/2}$, also allowing $k=-1$.
For $t>0$ and $z\in K$, define
\begin{equation}\label{eq:regularizer}
 J_t(z)=\min_{w\in K}
   \{\mathcal E(w)+t\,d(z,w)+t^2\rho(z,w)^2\}.
\end{equation}
A minimizer $\pi_tz$ exists by compactness. Choose one for each $t,z$, and set
\begin{equation}\label{eq:residuals}
 \varepsilon_k(z)=d(z,\pi_{t_k}z),\qquad
 \eta_k(z)=\rho(z,\pi_{t_k}z).
\end{equation}
The function $J_t(z)$ is nondecreasing in $t$ and lies in $[0,2]$.
Testing $J_{t_{k-1}}(z)$ at the minimizer for $t_k$ yields
\[
 J_{t_k}(z)-J_{t_{k-1}}(z)
 \ge(1-2^{-1/2})t_k\varepsilon_k(z)
       +\tfrac12t_k^2\eta_k(z)^2.
\]
Telescoping proves the uniform budget
\begin{equation}\label{eq:residual-budget}
 \sup_{z\in K}\sum_{k\ge0}
  \big(t_k\varepsilon_k(z)+t_k^2\eta_k(z)^2\big)\less1.
\end{equation}
Thus both the canonical approximation error and the change in derivative
profile have been charged before local covers are chosen.

\subsection{Localization inside one profile cell}
Choose nested admissible partitions $(\mathcal Q_k)$ of $K$ with
\begin{equation}\label{eq:Q-budget}
 \delta_k(z)=\diam_\rho\mathcal Q_k(z),\qquad
 \sup_z\sum_{k\ge0}2^k\delta_k(z)^2\less P^2.
\end{equation}
One may first use $(P+\epsilon)^2$ and let $\epsilon\downarrow0$.
For every $z\in K$, put
\begin{equation}\label{eq:lambda}
 \lambda_k(z)=e_k(T_zB_2^{m+n},\norm{\cdot}_2).
\end{equation}
By \eqref{eq:ellipsoid},
\begin{equation}\label{eq:lambda-budget}
 \sup_z\sum_{k\ge0}2^k\lambda_k(z)^2\less R^2.
\end{equation}

Fix $k\ge2$, a nonempty $B\subset K$, and one nonempty piece
$B_0=B\cap Q$ with $Q\in\mathcal Q_{k-2}$. Abbreviate
\begin{gather*}
 D=\diam_d B,\qquad \delta=\diam_\rho Q,\qquad t=t_k,\\
 \varepsilon=\sup_{z\in B_0}\varepsilon_k(z),\qquad
 \eta=\sup_{z\in B_0}\eta_k(z),\qquad A_t=\{\pi_tz:z\in B_0\}.
\end{gather*}
We claim that its Euclidean diameter $h=\diam_2 A_t$ satisfies
\begin{equation}\label{eq:local-radius}
 h^2\less t(D+\varepsilon)+t^2(\delta^2+\eta^2).
\end{equation}
Indeed, take $x=\pi_tz,y=\pi_tz'$ with $z,z'\in B_0$, and interchange them
if necessary so that $\mathcal E(x)\ge\mathcal E(y)$. Compare the minimum
in \eqref{eq:regularizer} at $x$ with the competitor $m=m(x,y)$.
By \eqref{eq:energy-mid},
\begin{align*}
 \tfrac12\norm{x-y}_2^2
 &\le\mathcal E(x)-\mathcal E(m)\\
 &\le t\big(d(z,m)-d(z,x)\big)
       +t^2\big(\rho(z,m)^2-\rho(z,x)^2\big).
\end{align*}
The first difference is at most $d(x,m)\le d(x,y)\le D+2\varepsilon$.
Also
\[
 \rho(z,m)\le\rho(z,x)+\rho(x,y)
 \le\delta+3\eta.
\]
Dropping the negative square proves \eqref{eq:local-radius}.

\subsection{Local tangent covers and fixed-point accounting}
Select $b\in B_0$ with $\lambda_{k-2}(b)$ arbitrarily close to
$\inf_{z\in B_0}\lambda_{k-2}(z)$. For every $x\in A_t$,
$\rho(x,b)\le\delta+\eta$. Fix $x_0\in A_t$.
The set $T_b(A_t-x_0)$ lies in $hT_bB_2^{m+n}$ and can be covered by at most
$2^{2^{k-2}}$ balls with radii arbitrarily close to
$h\lambda_{k-2}(b)$.

Choose one point of $A_t$ from each nonempty inverse cover cell. For two
points $x,y$ in the same cell, Lemma~\ref{lem:remainder} shows that
\[
 d(x,y)\le\norm{T_b(x-y)}_2+C(\delta+\eta)h.
\]
The chosen points therefore provide a cover of $F(A_t)$ at radius
$Ch(\lambda_{k-2}(b)+\delta+\eta)$. Each original point of $B_0$ is within
$\varepsilon$ of its image under $\pi_t$. Thus
\begin{equation}\label{eq:local-cover}
 e_{k-2}(B_0,d)
 \less\varepsilon+h(\lambda_{k-2}(b)+\delta+\eta).
\end{equation}
Infima in entropy radii or in the choice of $b$ can be approached and then
passed to a limit; no minimizing anchor is required.

Let $\ell=\inf_{z\in B_0}\lambda_{k-2}(z)$.
Insert \eqref{eq:local-radius} into \eqref{eq:local-cover} and use Young's
inequality. For any fixed $\theta>0$ this gives
\begin{equation}\label{eq:young}
 e_{k-2}(B_0,d)
 \le\theta D+C_\theta\big\{\varepsilon+
                         t(\ell^2+\delta^2+\eta^2)\big\}.
\end{equation}
For example, the term
$\sqrt{t(D+\varepsilon)}(\ell+\delta+\eta)$ is bounded by
$\theta(D+\varepsilon)+C_\theta t(\ell^2+\delta^2+\eta^2)$, and the
remaining product is bounded by the latter quadratic term.

There are at most $2^{2^{k-2}}$ profile pieces, each requiring at most that
many tangent-cover cells. Their total number is at most $2^{2^{k-1}}$, and
therefore at most $2^{2^k}$. Consequently \eqref{eq:young} implies
\begin{equation}\label{eq:entropy-contraction}
 e_k(B,d)\le\theta\diam_d B+C_\theta\sup_{z\in B}s_k(z),
\end{equation}
where, for $k\ge2$, we may take
\begin{equation}\label{eq:cost}
 s_k(z)=\varepsilon_k(z)+t_k\big(
     \lambda_{k-2}(z)^2+\delta_{k-2}(z)^2+\eta_k(z)^2\big).
\end{equation}
Here the constant accommodates the elementary inequality that a sum of a
fixed number of suprema is at most that number times the supremum of the
sum. Most importantly, $\ell\le\lambda_{k-2}(z)$ for \emph{every} original
point $z\in B_0$. Thus the local cover may use a different tangent anchor at
each level, but its charge is bounded by the entropy sequence of $T_z$ for
the same original $z$. There is no summation over changing anchors.

For $k=0,1$, take $s_k(z)=\diam_d K\le2a_*\le2R$.
Equations \eqref{eq:residual-budget}, \eqref{eq:Q-budget}, and
\eqref{eq:lambda-budget} give
\begin{align}
 \sup_z\sum_{k\ge0}2^{k/2}s_k(z)
 &\less R+a^{-1}+a(R^2+P^2).\label{eq:total-charge}
\end{align}
Indeed the $\varepsilon_k$ and $\eta_k$ contributions equal, up to the
factor $a^{-1}$, the two sums in \eqref{eq:residual-budget}. The other two
contributions use $2^{k/2}t_k=a2^k$ and a shift of two levels.

Take $\theta$ sufficiently small and apply the contraction principle
\eqref{eq:contraction}. This proves
\[
 \gamma_2(K,d)\less R+a^{-1}+a(R^2+P^2).
\]
For a nonzero matrix, choose $a=(R^2+P^2)^{-1/2}$; the zero matrix is
immediate. This proves Proposition~\ref{prop:transfer}.
The chaining functional is finite before applying contraction: $F$ is
Lipschitz on the compact finite-dimensional set $K$, so Euclidean covers
already give a finite, possibly dimension-dependent bound.

\begin{remark}[Why the hierarchy retains its refinements]
The recursive partition proof of the contraction principle
\cite[Section~3.3]{vH} refines the existing cells by the current cost labels
and then by the local cover. Past labels remain in the prefix. The profile
partitions in \eqref{eq:Q-budget} are nested from the outset. The
double-exponential admissibility allowance absorbs the product of the
previous numbers of cells after a fixed shift of levels. The three charges
in \eqref{eq:total-charge} have distinct, explicit sources: the two
regularization residuals telescope, the profile diameters use one fixed
hierarchy, and the tangent entropies use one fixed original point on each
branch. No extra unbounded support-label charge is left outside the estimate.
\end{remark}

\section{Completion of the Gaussian matrix bound}\label{sec:completion}
Let
\[
 X_z=\sum_{i,j}a_{ij}g_{ij}u_i^\alpha v_j^\beta,
 \qquad z=(u,v)\in K.
\]
This is a centered Gaussian process with canonical metric $d$, and its
index set contains a point with $X_z=0$. By the Gaussian upper bound,
Propositions~\ref{prop:profile} and \ref{prop:transfer}, and
Lemma~\ref{lem:trace},
\begin{equation}\label{eq:positive-width}
 \E\sup_{z\in K}X_z\less\gamma_2(K,d)
 \less R+P\less M.
\end{equation}
To pass to the full real operator norm, write any
$x\in B_{s^*}^m$, $y\in B_p^n$ as $x=x^+-x^-$ and $y=y^+-y^-$.
Each positive or negative part is still in the corresponding positive unit
ball. The bilinear form splits into four terms. If
$Z_+=\sup_{z\in K}X_z$ and $Z_-=\sup_{z\in K}(-X_z)$, then pointwise
\[
 \norm{G_A}_{p\to s}\le2Z_++2Z_-.
\]
Since $G_A$ and $-G_A$ have the same distribution, \eqref{eq:positive-width}
proves Theorem~\ref{thm:main}. All steps include $r=\infty$ or $s=\infty$;
the endpoint conventions were checked in the individual estimates.

\section{Prefix-dependent blocks and the corrected budget}\label{sec:blocks}
For completeness, we spell out the block consequence without identifying
the canonical image with a larger union of ellipsoids. Define the full
canonical coefficient set
\begin{equation}\label{eq:full-image}
 \mathcal T=
 \{(a_{ij}x_i y_j)_{i,j}:x\in B_{s^*}^m,\ y\in B_p^n\}
 \subset\R^{mn}.
\end{equation}
The four positive-part maps pull the chain on $F(K)$ back to the original
index set $B_{s^*}^m\times B_p^n$. Intersecting these four chains, with a
fixed level shift, and using the triangle inequality gives a nested
admissible chain for this index set in its canonical pseudometric with
budget at most $CM$. Equivalently, one may choose representative image
points, each from an original index, along these nested prefixes.

\begin{proposition}[Blocks indexed by prefixes]\label{prop:blocks}
There is a rooted, admissible hierarchy of original indices with a
rank-one block $A_v:\R\to\R^{mn}$ on each nonzero edge $v$ at level
$n(v)$, and nonnegative coefficients $c_v$, such that along the branch
$h(\theta)$ of every original index $\theta$,
\begin{align}
 \sum_{v\in h(\theta)}c_v&\le1,\label{eq:block-coeff}\\
 F_{\mathrm{full}}(\theta)&=
   \sum_{v\in h(\theta)}A_v\sqrt{c_v},\label{eq:block-image}\\
 \sup_h\sum_{v\in h}2^{n(v)}\norm{A_v}_\op^2&\less M^2.
 \label{eq:block-op}
\end{align}
Here $F_{\mathrm{full}}(x,y)=(a_{ij}x_i y_j)_{ij}$. The hierarchy can also
be chosen so that
\begin{equation}\label{eq:block-trace}
 \sup_h\sum_{v\in h}\norm{A_v}_\HS^2\less R^2.
\end{equation}
The supremum may include limit branches of the tree.
\end{proposition}
\begin{proof}
Take the nested chain just described, and at each node choose a
representative from its original cell. Set the root representative to
zero, which belongs to $\mathcal T$. Write $z_k(\theta)$ for the
representative image at level $k$. The diameter bound gives convergence to
$F_{\mathrm{full}}(\theta)$ and
\begin{equation}\label{eq:increments}
 \sup_\theta\sum_{k\ge1}2^{k/2}
    \norm{z_k(\theta)-z_{k-1}(\theta)}_2\le B_0,
 \qquad B_0\less M.
\end{equation}
Increase $B_0$ if necessary so that $B_0\ge R$. For a nonzero matrix,
choose an integer $q\ge1$ with
\[
 2^q\asymp(B_0/R)^2.
\]
Replace the initial representatives through level $q-1$ by zero, and retain
the original representatives from level $q$ onward. This amounts to
collapsing the initial levels of the hierarchy; admissibility is preserved.
Every point of $\mathcal T$ has Euclidean norm at most $a_*\le R$.
The new jump at level $q$ therefore has weighted size at most
$2^{q/2}R\less B_0$. The remaining increments are unchanged. Thus the
new sequence has branchwise weighted increment sum at most $B\less B_0$.

For a nonzero increment $\Delta_v$ on the edge $v$, define
\begin{equation}\label{eq:block-definition}
 c_v=\frac{2^{n(v)/2}\norm{\Delta_v}_2}{B},
 \qquad
 A_v(t)=\frac{t\Delta_v}{\sqrt{c_v}}.
\end{equation}
For a zero increment set $c_v=0$ and $A_v=0$. The definition gives
\eqref{eq:block-coeff}, while telescoping gives \eqref{eq:block-image}.
Since these blocks have one-dimensional domains,
\begin{align*}
 \sum_{v\in h}2^{n(v)}\norm{A_v}_\op^2
 &=B\sum_{v\in h}2^{n(v)/2}\norm{\Delta_v}_2\le B^2,\\
 \sum_{v\in h}\norm{A_v}_\HS^2
 &=B\sum_{v\in h}2^{-n(v)/2}\norm{\Delta_v}_2
 \le2^{-q}B^2\less R^2.
\end{align*}
All finite prefixes satisfy the same bounds, so they extend to limit
branches. The zero matrix has the trivial representation.
\end{proof}

\begin{remark}[Exact preservation of the image]
The equality in \eqref{eq:block-image} is asserted for the prescribed
coefficient sequence on each original branch. Its range is exactly
$\mathcal T$. If one also includes every infinite branch of the finitely
branching tree, its representative images converge to a point of
$\mathcal T$, since the tail bound is uniform and $\mathcal T$ is compact.
Allowing arbitrary vectors in the unit ball of the direct sum of block
domains would generally enlarge the image; no such equality is used or
claimed here. The weighted budget is $CM^2$, which includes the Gaussian
maximum term and both exponent-dependent row and column terms.
\end{remark}


\begin{thebibliography}{9}
\raggedright
\bibitem{LS}
R. Lata\l a and M. Strzelecka,
\emph{Operator $\ell_p\to\ell_q$ norms of Gaussian matrices},
arXiv:2502.02186v3, 9 June 2026.
\href{https://arxiv.org/abs/2502.02186v3}{arxiv.org/abs/2502.02186v3}.
The conjecture numbering used here is that of this version.

\bibitem{LT}
R. Lata\l a and T. Tkocz,
\emph{A note on suprema of canonical processes based on random variables
with regular moments},
arXiv:1406.6584v1, 2014, Theorem~1.
\href{https://arxiv.org/abs/1406.6584v1}{arxiv.org/abs/1406.6584v1}.

\bibitem{vH}
R. van Handel,
\emph{Chaining, interpolation, and convexity II: the contraction principle},
arXiv:1610.05199, 2016.
\href{https://arxiv.org/abs/1610.05199}{arxiv.org/abs/1610.05199}.
\end{thebibliography}
\end{document}